\documentclass{article}
\usepackage{harvard}
\usepackage{graphicx}
\usepackage{camnum}
\usepackage{amsmath}
\usepackage{amsfonts}
\usepackage{arydshln}

\usepackage{dutchcal}
\usepackage{listings}
\usepackage{mathalfa}
\usepackage{mathrsfs}

\usepackage{bm} 
\usepackage{fixmath} 

\usepackage{amsfonts}
\usepackage{graphicx}
\usepackage{epstopdf}
\usepackage{algorithmic}

\numberwithin{figure}{section} \numberwithin{equation}{section}

\allowdisplaybreaks

\newcommand{\G}{\CC{\Gamma}}
\newcommand{\ee}{{\mathrm e}}
\newcommand{\ii}{{\mathrm i}}

\newcommand{\PP}{\mathrm{P}}
\newcommand{\DDD}{\DD{D}}
\newcommand{\EEE}{\DD{E}}

\newcommand{\Ball}{\mathcal{B}}
\newcommand{\Azero}{{\mathfrak{A}^{\hspace*{-6pt}\raisebox{4pt}{\footnotesize$\circ$}}}}
\newcommand{\Dom}{\Omega}

\begin{document}
\title{A framework for stable spectral methods in $d$-dimensional unit balls}
\author{Jing Gao and Arieh Iserles}

\maketitle
\begin{abstract}
  The subject of this paper is the design of efficient and stable spectral methods for time-dependent partial differential equations in unit balls. We commence by sketching the desired features of a spectral method, which is defined by a choice of an orthonormal basis acting in the spatial domain. We continue by considering in detail the choice of a W-function basis in a disc in $\mathbb{R}^2$. This is a nontrivial issue because of a clash between two objectives: skew symmetry of the differentiation matrix (which ensures {\em inter alia\/} that the method is stable) and the correct behaviour at the origin. We resolve it by representing the underlying space as an affine space and splitting the underlying functions. This is generalised to any dimension $d\geq2$ in a natural manner and the paper is concluded with numerical examples that demonstrate how our choice of basis attains the best outcome out of a number of alternatives.
\end{abstract}

\tableofcontents

\thispagestyle{empty}

\section{Introduction: spectral methods}

We are concerned in this paper with the solution of differential equations of the form
\begin{equation}
  \label{eq:1.1}
  \frac{\partial u}{\partial t}=\mathcal{L}u+f(t,\MM{x},u),\qquad \MM{x}\in\Dom,\quad t\geq0,
\end{equation}
where $\Dom$ is a connected domain with Jordan boundary, given with $\CC{L}_2$ initial conditions in $\Dom$ and zero boundary conditions on its boundary. Here $\mathcal{L}$ is a linear differential operator that generates a strongly continuous semigroup, while $f$ is a Lipschitz function.  More specifically, our interest is focused on spectral methods \cite{funaro92pad,hesthaven07smt} in the special setting of $\Dom=\Ball$, where $\Ball\subset\mathbb{R}^d$ is the  unit ball,
\begin{displaymath}
  \Ball=\{\MM{x}\in\mathbb{R}^d\,:\, \|\MM{x}\|_2\leq1\}.
\end{displaymath}

While our setting is very general, throughout this paper we illustrate our ideas with two paradigmatic equations: the {\em diffusion equation\/}
\begin{displaymath}
  \frac{\partial u}{\partial t}=\Delta u
\end{displaymath}
and the {\em linear Schr\"odinger equation (LSE)\/}
\begin{displaymath}
  \ii\frac{\partial u}{\partial t}=-\Delta u+V(\MM{x})u.
\end{displaymath}

Assuming that the solution of \R{eq:1.1} evolves in the separable Hilbert space $\mathcal{H}$ with the inner product $\langle\,\cdot\,,\,\cdot\,\rangle$, we let $\Phi=\{\varphi_n\}_{n\in\mathbb{Z}_+}$ be an orthonormal basis of $\mathcal{H}$ that satisfies zero boundary conditions on $\partial\Dom$.  The solution of \R{eq:1.1} is expanded in the form
\begin{equation}
  \label{eq:1.2}
  u(\MM{x},t)=\sum_{m=0}^\infty \hat{u}_m(t)\varphi_m(\MM{x}),
\end{equation}
where the functions $\hat{u}_m$ are determined by the Galerkin conditions
\begin{displaymath}
  \left\langle \sum_{m=0}^\infty \hat{u}_m'\varphi_m-\mathcal{L} \sum_{m=0}^\infty \hat{u}_m\varphi_m-f\!\left(t,\,\cdot\,,\sum_{m=0}^\infty \hat{u}_m\varphi_m\right)\!,\varphi_n\right\rangle=0,\qquad n\in\mathbb{Z}_+.
\end{displaymath}
Because of orthonormality of $\Phi$ and linearity of $\mathcal{L}$, this reduces to the ODE system
\begin{equation}
  \label{eq:1.3}
  \hat{u}_n'=\sum_{m=0}^\infty \hat{u}_m \langle\mathcal{L}\varphi_m,\varphi_n\rangle +\left\langle f\!\left(t,\,\cdot\,,\sum_{m=0}^\infty \hat{u}_m\varphi_m\right)\!,\varphi_n\right\rangle,\qquad n\in\mathbb{Z}_+.
\end{equation}
Needless to say, in practice both \R{eq:1.2} and \R{eq:1.3} are truncated.

Let us assume that the $\varphi_n$s are sufficiently regular -- in  applications they are often analytic functions and this is the setting when spectral methods display their maximal advantage. Since $\Phi$ is a basis of $\mathcal{H}$ and $\varphi_n'\in\mathcal{H}$, there exists a linear transformation $\DDD^{\,[\ell]}$ taking the basis $\Phi$ to the  partial derivative with respect to $x_\ell$,
\begin{displaymath}
  \frac{\partial \varphi_m}{\partial x_\ell}=\sum_{n=0}^\infty \DDD_{m,n}^{\,[\ell]}\varphi_n,\qquad m\in\mathbb{Z}_+.
\end{displaymath}
The matrices $\DDD^{\,[\ell]}$, $\ell=1,\ldots,d$,  are called the {\em differentiation matrices\/} of $\Phi$.

This is the moment to list the desired features of a spectral method, while bearing in mind that the choice of a spectral method (at least in the current setting, we say nothing here about the solution of the semidiscretised ODE system \R{eq:1.3}) depends wholly on the choice of the orthonormal system $\Phi$. At the first instance, a spectral method (i.e., the choice of $\Phi$) must be assessed with respect to the following five desiderata:
\begin{enumerate}
\item {\bf Stability:\/} This is the {\em sine qua non\/} of any numerical method for time-dependent PDEs. By virtue of the Lax equivalence theorem, the solution of (truncated) \R{eq:1.3}  converges to $u$ as $N\rightarrow\infty$  if and only if the solution of 
\begin{displaymath}
    \hat{u}_n'=\sum_{m=0}^N \hat{u}_m \langle\mathcal{L}\varphi_m,\varphi_n\rangle,\qquad n=0,\ldots,N
\end{displaymath}
is uniformly well posed for $N\gg1$. 
\item {\bf Fast convergence:} In classical spectral methods the set $\Phi$ consists  of orthogonal polynomials or (in the case of periodic boundary conditions) is a Fourier basis. In that case, provided that all functions concerned are  analytic, the expansion \R{eq:1.2} converges at least at an exponential speed and we might choose relatively small $N$ while attaining high accuracy. This feature makes spectral methods very powerful {\em whenever they can be applied.\/} Unfortunately, this caveat typically restricts such methods to fairly simple geometries.
\item {\bf Fast expansion:} A major computational issue is to expand functions in the underlying basis $\Phi$, something that needs be done, often repeatedly, in each step of the solution. A major attraction of spectral methods in the presence of periodic boundary conditions is that Fourier bases lend themselves to fast computation with the Fast Fourier Transform but fast algorithms exist also for  expansions in the most popular bases of orthogonal polynomials \cite{olver29fau}.
\item {\bf Easy algebra:} Numerical mathematics reduces analysis to linear algebra, and this is the case with both the ODEs \R{eq:1.3} and their discretisation. Working at the level of orthogonal expansions, a function $g\in\mathcal{H}$ is represented by the coefficients of its orthogonal expansion $\hat{\MM{g}}=(\hat{g}_n)_{n\in\mathbb{Z}}$ such that $g=\sum_{n=0}^\infty \hat{g}_n \varphi_n$. We deduce that the representation of $\partial g/\partial x_\ell$ is $\DDD^{\,[\ell]}\hat{\MM{g}}$ and, with greater generality, the action of the differential operator $\mathcal{L}$ can be constructed, Lego blocks style, using differentiation matrices. Consequently, the structure of differentiation matrices is fundamental in ensuring that the underlying linear algebra -- not just representing the action of $\mathcal L$ by a matrix but solving linear systems and computing exponentials and other functions involving the matrix in question -- can be performed cheaply and effectively.
\item {\bf Preservation of structure:} Diverse PDEs possess structure which is of major importance in both  mathematics and their application areas. Thus, the $\CC{L}_2(\Dom)$ norm of the solution of the diffusion equation (with zero boundary conditions) should  monotonically decrease in $\CC{L}_2$: this reflects the second law of thermodynamics. In the `spectral method world' and using Cartesian coordinates this means that $\sum_{\ell=1}^d {\DDD^{\,[\ell]}}^*\DDD^{\,[\ell]}$ should be a negative semidefinite matrix.

Likewise, the solution of LSE is {\em unitary:\/} The $\CC{L}_2(\Dom)$ norm stays constant for all $t\geq0$. This reflects the quantum-mechanical interpretation of LSE: for any measurable subset $\Xi\subseteq\Dom$ the integral $\int_\Xi |u(\MM{x},t)|^2\D \MM{x}$ is the probability of a particle residing in the set $\Xi$ at time $t\geq0$. Thus, unless $\|u(\,\cdot\,,t)\|\equiv1$, the solution is nonphysical. In the present context this means that $\ii \sum_{\ell=1}^d {\DDD^{\,[\ell]}}^*\DDD^{\,[\ell]}$ should be a skew-Hermitian matrix.

Note that both conditions -- negative definiteness and skew-Hermicity -- imply stability, thereby fulfilling the first, most important feature of any method for time-dependent PDEs. Note further that, once each $\DDD^{[\ell]}$ is skew-Hermitian then $\sum_{\ell=1}^d {\DDD^{\,[\ell]}}^*\DDD^{\,[\ell]}$ is negative semidefinite and $\ii \sum_{\ell=1}^d {\DDD^{\,[\ell]}}^*\DDD^{\,[\ell]}$ skew Hermitian. This underlies the {\em leitmotif\/} of this paper: {\em using orthonormal systems with skew-Hermitian differentiation matrices.\/}

There are many other structures of interest, not least Hamiltonians, but they form no part of the work of this paper.
\end{enumerate}

It is important to comment on a critical issue. $\ii \sum_{\ell=1}^d {\DDD^{\,[\ell]}}^*\DDD^{\,[\ell]}$ is skew Hermitian for {\em any\/} square matrices $\DDD^{\,[\ell]}$ but, in our context, we need each $-{\DDD^{[\ell]}}^*$ to be a differentiation matrix as well. Therefore, skew Hermicity bridges the strong and the weak form of the PDE and is crucial in approximating the LSE while  maintaining the $\CC{L}_2$ norm and approximating the diffusion equation while dissipating it.

An obvious orthonormal set consists of polynomials: their univariate theory is exceedingly well known \cite{chihara78iop,ismail05cqo} and has been extended to numerous multivariate sets, in particular to discs \cite{zernike34bss} and simplexes \cite{dunkl14ops,koornwinder75tva}. However, while orthogonal polynomials have been used extensively in spectral methods for boundary-value problems, they are highly problematic insofar as initial-value PDEs \R{eq:1.1} are concerned: their differentiation matrices are lower-triangular. This is inimical to preservation of unitarity and, moreover, leads to instability. While stability can be recovered  using in tandem two orthonormal systems \cite{townsend15asp}, our approach is to abandon polynomial bases altogether.

In \cite{iserles19oss,iserles20for,iserles21fco,iserles24ost} we have introduced for $d=1$ and $\Dom=\mathbb{R}$ the concept of {\em T-functions.\/} For any Borel measure supported on $\mathbb{R}$ we can construct, using a Fourier transform, a corresponding set $\Phi$ of functions for which the differentiation matrix $\DDD$ is skew-Hermitian and tridiagonal. Skew-Hermicity ensures unitarity (hence also stability) and tridiagonality is a gateway to fast algebra. The univariate setting can be extended to $\Dom=\mathbb{R}^d$ using tensor products. However, T-functions are unsuitable in the presence of finite boundaries. This motivated in \cite{iserles24ssm} the introduction of {\em W-functions.\/}

Given an interval $(a,b)$ and a $\CC{C}^1(a,b)$ weight function $w(x)\geq0$, $\int_a^b w(x)\D x>0$, we let $\{p_n\}_{n\in\mathbb{Z}_+}$ be the underlying set of orthonormal polynomials. The $n$th W-function is then $\varphi_n(x)=\sqrt{w(x)} p_n(x)$, $n\in\mathbb{Z}_+$ and it is easy to see that $\Phi$ is an orthonormal basis of $\mathcal{H}=\CC{L}_2(a,b)$. Moreover, it has been proved in \cite{iserles24ssm} that the underlying differentiation matrix $\DDD$ is skew symmetric if and only if $w(a)=w(b)=0$. Two cases have been worked out in great detail in \cite{iserles24ssm}: {\em ultraspherical W-functions,\/} generated by $w(x)=(1-x^2)^\alpha$, $x\in(-1,1)$, for $\alpha>0$ and {\em Laguerre W-functions,\/} generated by $w(x)=x^\alpha \ee^{-x}$, $x\in(0,\infty)$, again for $\alpha>0$. In both instances, while not being  tridiagonal, $\DDD$ has another welcome feature: it is semi-separable of rank s \cite{hackbusch15hma}, where $s=1$ for Laguerre W-functions and $s=2$ for ultraspherical W-functions: every submatrix placed entirely above (or entirely below) the main diagonal is of rank s. This allows for fast algebra. Moreover, judicious choice of $\alpha$ leads to exponential convergence.

W-functions can be extended to $d$-dimensional boxes by tensor products. In that instance the differentiation matrix (in the ultraspherical or Laguerre cases) is semi-separable and the advantage of fast algebra is maintained.

A word on the function spaces used in this paper. The least requirement one may impose is that the underlying functions live in a (possibly weighed) $\CC{L}_2$ space in $\Ball$, but this is way too weak both for meaningful statements on convergence in the context of spectral methods and for the imposition of boundary conditions. In this paper we consider functions analytic in $\Ball$ (considered as a subset of $\mathbb{R}^2$), i.e.\ whose Taylor expansion about the origin converges in $\Ball$: we denote this function space by $\mathfrak{A}(\Ball)$. The intermediate case, the Sobolev space $\CC{H}^p_2(\Ball)$ for a suitable $p\geq1$, follows at once with obvious limitations, e.g.\ exponential convergence reduces to  polynomial one. In addition, we assume that all underlying functions obey zero boundary conditions on the unit sphere.

In Section~2 we focus on the disc, $d=2$. Interestingly, this simplest nontrivial case already displays all the intricacies associated with designing orthonormal systems in any dimension $d\geq2$. The obvious course of action is to use a scalar inner product in polar coordinates $0\leq r\leq 1$, $-\pi\leq\theta_1\leq\pi$, $0\leq\theta_2,\ldots,\theta_{d-1}\leq\pi$
\begin{equation}
  \label{eq:1.4}
  \langle\!\langle f,g\rangle\!\rangle=\int_{\|\Mm{x}\|\leq1} f(\MM{x})\overline{g(\MM{x})}\D \MM{x}=\int_0^1\int_{-\pi}^\pi \int_0^\pi\cdots\int_0^\pi  r f(r,\MM{\theta})\overline{g(r,\MM{\theta})}\D \MM{\theta}\D r,
\end{equation}
which is the same as standard Cartesian $\CC{L}_2$ inner product for $x^2+y^2\leq1$. This is the standard approach to approximation in a disc and there exists rich theory on polynomial approximation, using Zernike polynomials \cite{zernike34bss} and their generalisations \cite{wunsche05gzd}. It might thus seem that all we need is to `translate' polynomials to W-functions {\em \'a la\/} \cite{iserles24ssm} but, unfortunately, this course of action leads to systems whose differentiation matrices depart substantially from skew symmetry. 

As an alternative, we are using a Cartesian inner product in the box $\{(r,\theta)\,:\, r\in[0,1],\;\theta_1\in[-\pi,\pi],\; \theta_j\in[0,\pi],\; j=2,\ldots,d-1\}$, 
\begin{equation}
  \label{eq:1.5}
  \langle f,g\rangle=\int_0^1\int_{-\pi}^\pi \int_0^\pi\cdots\int_0^\pi f(r,\MM{\theta})\overline{g(r,\MM{\theta})}\D \MM{\theta}\D r,
\end{equation}
but this is not problem-free either. The boundary conditions on $\|\MM{x}\|=1$ translate seamlessly to $r=1$, while the boundary conditions on $\MM{\theta}$ remain periodic. However, for $r=0$ we require natural boundary conditions, $\partial f(0,\MM{\theta})/\partial\theta_m\equiv0$, $m=1,\dots,d-1$ \cite[p.~336--7]{iserles09fna}. This is inconsistent with our approach to W-functions: if $\varphi_{n,m}(r,\theta)=\tilde{\varphi}_n(r)\ee^{\ii m\theta}$ is a W-function then $\partial\varphi_{n,m}(0,\theta)=\ii m \tilde{\varphi}_n(0)\ee^{\ii m\theta}$ and setting this to zero for all $n$ implies $\varphi_{n,m}(0,\theta)=0$ for all $n\in\mathbb{Z}_+$. This precludes convergence in the $\CC{L}_\infty$ norm and renders $\CC{L}_2$ convergence exceedingly slow.

Yet, imposing $\varphi_{n,m}(0,\theta)=0$ has one beneficial effect: it represents a necessary and sufficient condition, similarly to \cite{iserles24ssm}, for differentiation matrices to be skew symmetric. The tension between skew symmetry and fast convergence is at the heart of this paper: a course of action, which we advocate in Section~2, is to split $f\in\mathfrak{A}(\Ball)$ into a sum of the form $f_0+f_1$ such that $f_0$ is orthogonal to $f_1$ and $f_1(0,\MM{\theta})=0$ -- in addition, $f_0(1,\MM{\theta})=f_1(1,\MM{\theta})=0$. In other words, we consider $\mathfrak{A}(\Ball)$ with zero boundary condition for $r=1$ as an affine space, $\mathfrak{A}(\Ball)=h(r,\theta)\oplus \Azero(\Ball)$, where $h\in\mathfrak{A}(\Ball)$, $h(0,\MM{\theta})=f(0,\MM{\theta})$, while $\Azero(\Ball)\subset\mathfrak{A}(\Ball)$ consists of functions that obey zero boundary conditions at $r=0$. Using a variation upon the theme of Zernike polynomials, in tandem with the W-function approach, we present an orthonormal basis on $\mathfrak{A}(\Ball)$ whose differentiation matrices in $\Azero(\Ball)$ are skew symmetric: we say that such differentiation matrices are {\em essentially skew symmetric.\/} Since such matrices have at most a finite number of eigenvalues in the complex right-half plane, they allow for the construction of stable semidiscretisations for our two paradigmatic equations.

In Section 2 we develop in detail the theory in the disc, $d=2$, and this is generalised in a straightforward manner to any $d\geq2$ in Section~3. Finally, in Section~4 we subject the same function $f$ to five different computational approaches based on the ideas in this paper. Four of these approaches do not tick {\em all\/} the boxes and they either lead to slow convergence or to differentiation matrices that depart significantly from skew symmetry, while the last  is the one we are advocating and this is confirmed by numerical results.

\section{Zernike-type W-functions}

\subsection{$d=2$}

In this paper we consider W-functions on unit balls $\Ball\subset\mathbb{R}^d$ with zero boundary conditions on $\partial\Ball$. In the present section we commence with $d=2$, hence, in polar coordinates, a single radial variable $r\in[0,1]$ and a single angular variable $\theta\in[-\pi,\pi]$,  imposing zero boundary conditions for $r=1$. The usual $\CC{L}_2$ inner product in polar coordinates has been given in \R{eq:1.4} and it  motivates the standard orthogonal system on $\Ball$ in polar coordinates, namely {\em Zernike polynomials\/}
\begin{equation}
  \label{eq:2.1}
  p_{n,m}(r,\theta)=\PP_n^{(0,1)}(2r-1)\ee^{\ii m\theta},\qquad r\in[0,1],\; \theta\in[-\pi,\pi],\qquad n\in\mathbb{Z}_+,\; m\in\mathbb{Z},
\end{equation}
where $\PP_n^{(\alpha,\beta)}$ is an $n$th-degree Jacobi polynomial \cite{zernike34bss} (the original system is usually written somewhat differently, but is isomorphic to \R{eq:1.3}). Note that the polynomials $\{\PP_n^{(\alpha,\beta)}\}_{n\in\mathbb{Z}_+}$ are orthogonal in $[-1,1]$ with respect to the weight function $(1-x)^\alpha(1+x)^\beta$ and this renders Zernike polynomials orthogonal in polar coordinates.

Zernike polynomials been extended to the so-called {\em disc polynomials\/} in \cite{wunsche05gzd}, allowing for more general Jacobi weight functions. However, \R{eq:2.1} and its generalisation to disc polynomials exhibit two critical shortcomings in the present context: firstly, their differentiation matrices are not skew symmetric and, secondly, they fail to obey zero boundary conditions for $r=1$.

With greater generality, the inner product \R{eq:1.4} is highly problematic within our framework: consider just functions of the radial variable $r$ and set with greater generality
\begin{equation}
  \label{eq:2.2}
  \langle\!\langle f,g\rangle\!\rangle=\int_0^1w(r) f(r)\overline{g(r)}\D r,
\end{equation}
where $w$ is a differentiable weight function, $w\not\equiv\CC{const}$.

\begin{theorem}
  Let a scalar product be given by \R{eq:2.2} . The conditions
  \begin{enumerate}
  \item[(a)] $\{q_n\}_{n\in\mathbb{Z}_+}$ is a real orthogonal basis with respect to the inner product $\langle\!\langle\,\cdot\,,\,\cdot\,\rangle\!\rangle$;
  \item[(b)] $w(0)=0$, while $w(x)\neq0$, $x\in(0,1)$;
  \item[(c)] $q_n(1)=0$, $n\in\mathbb{Z}_+$;
  \item[(d)] The differentiation matrix
  \begin{displaymath}
    \DDD_{m,n}=\langle\!\langle q_m',q_n\rangle\!\rangle,\qquad m,n\in\mathbb{Z}_+,
  \end{displaymath}
  is skew Hermitian
  \end{enumerate}
  are contradictory.
\end{theorem}

\begin{proof}
  Let $m,n\in\mathbb{Z}_+$. Integrating by parts,
  \begin{displaymath}
    \DDD_{m,n}=\langle\!\langle q_m',q_n\rangle\!\rangle= w(r)q_m(r)q_n(r)\,\rule[-4pt]{0.75pt}{18pt}_{\,r=0}^{\, r=1}- \int_0^1 w' q_mq_n\D r-\int_0^1 wq_mq_n'\D r
  \end{displaymath}
therefore, because of conditions (b) and (c),
  \begin{displaymath}
    \DDD_{m,n}+\DDD_{n,m}=-\int_0^1 w'q_mq_n\D r.
  \end{displaymath}
  It now follows from condition (d) that $\int_0^1 w'q_mq_n\D r=0$ for all $m,n\in\mathbb{Z}_+$, but this is clearly impossible. Otherwise, since $\{\varphi_n\}_{n\in\mathbb{Z}_+}$ is an orthogonal basis and every function in the underlying Hilbert space $\mathcal{G}$ can be expanded in it, it would have followed that $\langle\!\langle w'/w,f\rangle\!\rangle=0$ for all $f\in\mathcal{G}$. This is possible only if $w'\equiv0$ but this is ruled out by condition (b).
\end{proof}

The theorem is not very surprising. The deep reason skew symmetry matters is its connection to standard $\CC{L}_2$ unitarity: the Lie algebra corresponding to the unitary group $\CC{U}(n,\mathbb{C})$ is $\GG{su}(n,\mathbb{C})$, the set of skew-Hermitian matrices. Yet, the inner product \R{eq:1.4} is not a `standard' $\CC{L}_2$! Therefore it makes better sense to seek convergence (and to measure the error) not in  the  inner product \R{eq:1.4}, natural in polar coordinates, and as an alternative  employ the standard $\CC{L}_2$ inner product \R{eq:1.5}. Thus, we convert \R{eq:2.1} (or, with greater generality, disc polynomials) to W-functions, letting
\begin{equation}
  \label{eq:2.3}
  \varphi_{n,m}(r,\theta)=\pi^{-1/2} 2^{\frac12(\alpha+\beta)}(1-r)^{\alpha/2} r^{\beta/2} \tilde{\PP}_n^{(\alpha,\beta)}(2r-1)\ee^{\ii m\theta}, \qquad n\in\mathbb{Z}_+,\; m\in\mathbb{Z},
\end{equation}
where $\alpha,\beta>-1$ and
\begin{displaymath}
  \tilde{\PP}_n^{(\alpha,\beta)}(x)=\frac{\PP_n^{(\alpha,\beta)}(x)}{\sqrt{h^{(\alpha,\beta)}_n}},\qquad h^{(\alpha,\beta)}_n=\frac{2^{1+\alpha+\beta}\G(1+\alpha+n)\G(1+\beta+n)}{n!(1+\alpha+\beta+2n)\G(1+\alpha+\beta+n)},
\end{displaymath}
are orthonormal Jacobi polynomials \cite[p.~260]{rainville60sf}.\footnote{The factor $2^{\frac12(\alpha+\beta)}$ is due to the change of variable $x=2r-1$ and division by $\sqrt{2}$.} Yet, this choice is problematic as well. In order for $\DDD$ to be skew symmetric, a necessary and sufficient condition is $\alpha,\beta>0$ \cite{iserles24ssm}. However $\beta>1$ implies that $\varphi_{m,n}(0,\theta)\equiv0$: this precludes $\CC{L}_\infty$ convergence for general functions and indicates that $\CC{L}_2$ convergence, if at all possible, would be excruciatingly slow.

To avoid this conundrum, we split $f\in\mathfrak{A}(\Ball)$ in the form $f=f_0+f_1$, subject to three conditions:
\begin{enumerate}
\item $f_0(1)=f_1(1)=0$ (boundary conditions for $r=1$ are satisfied);
\item $f_0(0)=f(0)$, $f_1(0)=0$;
\item $\langle f_0,f_1\rangle=0$, therefore $\|f\|^2=\|f_0\|^2+\|f_1\|^2$.
\end{enumerate}
Given $f\in\CC{C}(\Ball)$, we call a pair $\{f_0,f_1\}$ of $\mathfrak{A}(\Ball)$ functions such that $f_0+f_1=f$ and conditions 1--3 are satisfied, a {\em proper orthogonal splitting (POS).\/} Herewith a constructive means to produce a POS.

Let $\tilde{f}_0$ and $\tilde{f}_1$ be two $\mathfrak{A}(\Ball)$ functions that obey the conditions 1--2 and such that $\tilde{f}_0+\tilde{f}_1=f$. We set, in a single step of the Gram--Schmidt algorithm,
\begin{equation}
  \label{eq:2.4}
  f_0(r)=\tilde{f}_0(r)-c\tilde{f}_1(r),\qquad f_1(r)=(1+c)\tilde{f}_1(r),\qquad\mbox{where}\qquad c=\frac{\langle \tilde{f}_0,\tilde{f}_1\rangle}{\|\tilde{f}_1\|^2}.
\end{equation}
It is trivial to prove that conditions 1--3 hold and we have a POS.

Reasonable choices of $\tilde{f}_0$ and $\tilde{f}_1$ in \R{eq:2.4} are $\tilde{f}_0(r)=f(0)(1-r)$ or $\tilde{f}_0(r)=f(0)\cos(\pi r/2)$, with $\tilde{f}_1(r)=f(r)-\tilde{f}_0(r)$.

The idea is to expand a general $f\in\mathfrak{A}(\Ball)$ in the form
\begin{equation}
  \label{eq:2.5}
  f(r,\theta)=\sum_{m=-\infty}^\infty f^\circ_m \ee^{\ii m\theta}+\sum_{n=0}^\infty \sum_{m=-\infty}^\infty \hat{f}_{n,m} \varphi_{n,m}(r,\theta),
\end{equation}
where
\begin{Eqnarray*}
  f_m^\circ&=&(2\pi)^{-1/2} \int_{-\pi}^\pi f_0(r,\theta)\ee^{-\ii m\theta}\D\theta,\\
  \hat{f}_{n,m}&=& \int_0^1\int_{-\pi}^\pi f_1(r,\theta)\overline{\varphi_{n,m}(r,\theta)}\D\theta\D r,\qquad n\in\mathbb{Z}_+,\;\; m\in\mathbb{Z},
\end{Eqnarray*}
and $\varphi_{n,m}$ has been given in \R{eq:2.3}.

As would become clear in Subsection 2.3, differentiation matrices arising from the expansion of $f_1$ are skew symmetric, as is the differentiation matrix of $f_0$ with respect to $\theta$. However, it is impossible to expect the $1\times1$ `differentiation matrix' of $f_0$ with respect to $r$ to be skew symmetric unless $f_0\equiv0$. This is not a serious impediment to the use of our W-functions and, as will see, both diffusion and linear Schr\"odinger equations can be solved in a stable manner with the present approach, while the solution of LSE is unitary.

\subsection{The optimal choice of $\alpha,\beta>0$}

Let $\DDD^{[r]}$ be a differentiation matrix with respect to $r$ only.

In principle, any choice of $\alpha$ and $\beta$ in \R{eq:2.3} would do, as long as they are both positive, to ensure that the differentiation matrix is skew symmetric. In practice, though, we must consider three imperatives:
\begin{enumerate}
\item Fast convergence of the expansion \R{eq:2.5};
\item Explicit knowledge of the entries of $\DDD^{[r]}$; and
\item Fast algebra.
\end{enumerate}

The first condition is satisfied when both $\alpha$ and $\beta$ are even integers, since then $\varphi_{n,m}$ is an analytic function and, using similar construction to \cite{iserles24ssm} and standard results on the convergence of analytic orthogonal sequences from \cite{walsh65iar}, it follows that the speed of convergence in \R{eq:2.5} is exponential, as long as $f$ is analytic for $r\in(a,b)$, $a<0$ and $1<b$, while having a simple zero at both $r=0$ and $r=1$.

The entries of $\DDD^{[r]}$ are known explicitly in the ultraspherical case, i.e.\ when $\alpha=\beta$,
\begin{equation}
  \label{eq:2.6}
  \DDD_{m,n}^{[r]}=
  \begin{case}
    \GG{a}_m\GG{b}_n, & m\geq n+1,\; m+n\mbox{\ odd},\\
    -\GG{a}_n\GG{b}_m, & m\leq n-1,\; m+n\mbox{\ odd},\\
    0, & m+n\mbox{\ even}
  \end{case}
\end{equation}
where
\begin{displaymath}
  \GG{a}_m=\sqrt{\frac{m!(2m+2\alpha+1)}{2\G(m+1+2\alpha)}},\qquad \GG{b}_n=\sqrt{\frac{(2n+1+2\alpha)\G(n+1+2\alpha)}{2n!}},\qquad m,n\in\mathbb{Z}_+.
\end{displaymath}
\cite{iserles24ssm}. More specifically, for $\alpha=\beta=2$ we have
\begin{equation}
  \label{eq:2.7}
  \GG{a}_m=\sqrt{\frac{m!(2m+5)}{2(m+4)!}},\qquad \GG{b}_n=\sqrt{\frac{(2n+5)(n+4)!}{2n!}}.
\end{equation}

Therefore, $\DDD^{[r]}$ is semi-separable of rank 2 -- as a matter of fact, given that $\DDD_{m,n}=0$ once $m+n$ is even, the cost of fast algebra is practically the same as for semi-separability of rank 1. Thus, the third condition is satisfied as well.

\subsection{Differentiation matrices}

It is convenient to separate between differentiation matrices with respect to $r$ and $\theta$:
\begin{displaymath}
  \DDD^{\,[r]}_{(n,m),(k,\ell)}= \int_0^1\int_{-\pi}^\pi \frac{\partial\varphi_{n,m}(r,\theta)}{\partial r}\overline{\varphi_{k,\ell}(r,\theta)}\D\theta\D r,\qquad n,k\in\mathbb{Z}_+,\quad m,\ell\in\mathbb{Z}
\end{displaymath}
and
\begin{displaymath}
  \DDD^{\,[\theta]}_{(n,m),(k,\ell)}= \int_0^1\int_{-\pi}^\pi \frac{\partial\varphi_{n,m}(r,\theta)}{\partial \theta}\overline{\varphi_{k,\ell}(r,\theta)}\D\theta\D r,\qquad n,k\in\mathbb{Z}_+,\quad m,\ell\in\mathbb{Z}.
\end{displaymath}
Recall that we have decomposed $f=f_0+f_1$: it makes sense to consider first the application of differentiation matrices to $f_1$.

It follows at once from \R{eq:2.3} that
\begin{displaymath}
  \DDD^{\,[r]}_{(n,m),(k,\ell)}=\DDD^{\,[\theta]}_{(n,m),(k,\ell)}=0,\qquad m\neq\ell,
\end{displaymath}
while (recalling that $\beta=\alpha$ and using \R{eq:2.6})
\begin{Eqnarray*}
  \DDD^{\,[r]}_{(n,m),(k,m)}&=&2^{2\alpha+1} \int_0^1(1-r)^\alpha r^\alpha \frac{\D\tilde{\PP}_n^{(\alpha,\alpha)}(2r-1)}{\D r} \tilde{\PP}_k^{(\alpha,\alpha)}(2r-1)\D r\\
  &=&\int_{-1}^1 (1-x^2)^\alpha \frac{\D\tilde{\PP}_n^{(\alpha,\alpha)}(x)}{\D x}\tilde{\PP}_k^{(\alpha,\alpha)}(x)\D x\\
  &=&\begin{case}
    \GG{a}_n\GG{b}_k, & k\geq n+1,\; n+k\mbox{\ odd},\\
    -\GG{a}_k\GG{b}_n, & k\leq n-1,\; n+k\mbox{\ odd},\\
    0, & n+k\mbox{\ even},
  \end{case}
\end{Eqnarray*}
where the $\GG{a}_m$s and $\GG{b}_m$s have been given in \R{eq:2.7}. Likewise,
\begin{Eqnarray*}
  &&\DDD^{\,[\theta]}_{(n,m),(k,m)}\\
  &=&2^{2\alpha+1} (2\pi)^{-1} \int_0^1 (1-r)^\alpha r^\alpha \tilde{\PP}_n^{(\alpha,\alpha)}(2r-1)\tilde{\PP}_k^{(\alpha,\alpha)}(2r-1) \int_{-\pi}^\pi \frac{\D \ee^{\ii m\theta}}{\D\theta} \ee^{-\ii m\theta}\D\theta \\
  &=&\begin{case}
    \ii m, & n=k,\\[2pt]
    0, & n\neq k.
  \end{case}
\end{Eqnarray*}

We deduce not just that $\DDD^{\![r]}$ and $\DDD^{\,[\theta]}$ are both skew Hermitian but that they are massively sparse. This helps fast algebra a great deal, as does the fact that the nonzero blocks of $\DDD^{\,[r]}$ are semi-separable of rank 2, while the nonzero blocks of $\DDD^{\,[\theta]}$ are diagonal.

The function $f_0$ leads to a diagonal, skew-Hermitian $\DDD^{\,[\theta]}$, indexed over $\mathbb{Z}$ and to a $1\times1$ matrix $\DDD^{\,[r]}$ which cannot be skew symmetric unless $f_0(0,\theta)\equiv0$:
\begin{Eqnarray*}
  \label{eq:2.8}
  \Re\int_0^1 \int_{-\pi}^\pi \frac{\partial f_0(r,\theta)}{\partial r} \overline{f_0(r,\theta)}\D\theta\D r&=&\frac12\int_{-\pi}^\pi |f_0(r,\theta)|^2\,\rule[-4pt]{0.75pt}{18pt}_{\,r=0}^{\, r=1}\D\theta=-\frac12 \int_{-\pi}^\pi |f_0(0,\theta)|^2\D\theta\\
  &=&-\frac12 \int_{-\pi}^\pi |f(0,\theta)|^2\D\theta,
\end{Eqnarray*}
because $f_0(1,\theta)\equiv0$.

\vspace{6pt}
\noindent {\bf Definition 1} {\em We say that a differentiation matrix is {\em essentially skew symmetric\/} if its action on $f_1$ is skew symmetric.}

\vspace{6pt}
We deduce that, once acting in Cartesian coordinates, the differentiation matrix $\{\DDD^{\,[r]},\DDD^{[\,\theta]}\}$ is essentially skew symmetric.

The Laplacian in polar coordinates being
\begin{displaymath}
  \Delta f=\frac{\partial^2f}{\partial r^2}+\frac{1}{r}\frac{\partial f}{\partial r}+\frac{1}{r^2}\frac{\partial^2f}{\partial\theta^2},
\end{displaymath}
integration by parts, in tandem with zero boundary conditions for $r=1$, yield
\begin{Eqnarray*}
  \langle\!\langle \Delta f,f\rangle\!\rangle&=&\int_0^1\int_{-\pi}^\pi r \!\left(\frac{\partial^2 f}{\partial r^2}+\frac{1}{r}\frac{\partial f}{\partial r}+\frac{1}{r^2} \frac{\partial^2f}{\partial\theta^2}\right)\!\bar{f}\D\theta\D r\\
  &=&-\int_0^1\int_{-\pi}^\pi \frac{\partial |f|^2}{\partial r}\D\theta\D r -\int_0^1\int_{-\pi}^\pi \left(r \left|\frac{\partial f}{\partial r}\right|^{\!2} +\frac{1}{r} \left|\frac{\partial f}{\partial\theta}\right|^{\!2}\right)\D\theta\D r.
\end{Eqnarray*}
Therefore, negative semidefiniteness with respect to the inner product \R{eq:1.4} is not assured and it is easy to check that the same applies to the inner product \R{eq:1.5}. Yet, the two scalar products are `almost' negative semidefinite: in particular, certain beneficial aspects of negative semidefiniteness remain valid in this setting.

Let us examine the impact of this observation on the stability in the case of the diffusion equation and the LSE. To this end we consider {\em compound differentiation matrices\/} $\EEE^{[r]}$ and $\EEE^{[\theta]}$. Their first row and column originate in $f_0$, while the $(1,1)$ principal minor stem from $f_1$. Thus,
\begin{displaymath}
  \EEE^{[r]}=\left[
  \begin{array}{cc}
         d & \MM{0}^\top\\
         \MM{0} & \DDD^{[r]}
  \end{array}
  \right]\!,\qquad \Re d=-\frac12 \int_{-\pi}^\pi  |f(0,\theta)|^2\D\theta,
\end{displaymath}
while $\EEE^{[\theta]}$ is skew-Hermitian. It follows that $\mathcal{R}={\EEE^{[r]}}^\top\!\EEE^{[r]}$ is symmetric and its largest eigenvalue is bounded, while $\Theta={\EEE^{[\theta]}}^*\!\EEE^{[\theta]}$ is symmetric and negative semidefinite. Therefore $\|\exp(t(\mathcal{R}+\Theta))\|\leq \ee^{|d|^2t}$ and the method is stable for the diffusion equation (essentially, because $\mathcal{R}$ is symmetric and its spectrum is bounded in the complex right half-plane -- in the language of semigroup theory, it is a sectorial operator). Note however that dissipativity might be lost.

Matters are simpler in the case of LSE. $\mathcal{R}+\Theta$ being symmetric, it follows that $\ii(\mathcal{R}+\Theta)$ is skew-Hermitian, therefore $\|\exp(\ii t(\mathcal{R}+\Theta))\|\equiv1$: the method is both stable and preserves unitarity.

\subsection{Computational considerations}

There are two types of computational issues that must be taken into account while implementing W-functions in the context of spectral methods for PDE. Firstly, we must consider the purely approximation-theoretic issue of representing functions $f\in\mathfrak{A}(\Ball)$ in the underlying basis and, secondly, linear-algebraic considerations occurring once these expansions are applied in a spectral method.

At the outset we need to decide on the value of $\alpha$: the choice $\alpha=2$ appears to be optimal since it leads to exponential convergence for analytic functions $f$ and imposes the right boundary condition for $r=1$ \cite{iserles24ssm}.

Given $f\in\mathfrak{A}(\Ball)$, represented in polar coordinates, we expand
\begin{equation}
  \label{eq:2.9}
  \hat{f}_{m,n}=\int_0^1\int_{-\pi}^\pi f(r,\theta)\overline{\varphi_{m,n}(r,\theta)}\D\theta\D r,\qquad m=0,\ldots,M,\quad |n|=0,\ldots,N.
\end{equation}
Practical computation of the $\theta$ integral requires Fast Fourier Transform at the cost of $\O{N\log_2N}$ operations. The $r$ integral can be computed by any of the $\O{M\log_2M}$ methods in \cite{cantero12rce,olver29fau}. Of course, in practice we need first a POS $f=f_0+f_1$, while applying \R{eq:2.9} to $f_1$ -- expanding $f_0$ can be accomplished by a single FFT.

Next is the construction of  differentiation matrices. We need first to evaluate the coefficients $\GG{a}_m,\GG{b}_m$ in \R{eq:2.6} for $m=0,\ldots,M$. This can be done recursively in $\O{M}$ operations:
\begin{Eqnarray*}
  \GG{a}_0&=&\sqrt{\frac{2\alpha+1}{2\G(2\alpha+1)}},\qquad \GG{a}_m=\GG{a}_{m-1}\sqrt{\frac{m(2m+2\alpha+1)}{(m+2\alpha)(2m+2\alpha-1)}},\quad m=1,\ldots,M,\\
  \GG{b}_0&=&\sqrt{\frac{\G(2\alpha+2)}{2}},\qquad \GG{b}_m=\GG{b}_{m-1}\sqrt{\frac{(2m+1+2\alpha)(m+2\alpha)}{m(2m+2\alpha-1)}},\quad m=1,\ldots,M
\end{Eqnarray*}
and need be done just once, rather than in each time step. 

Cognisant of the discussion in Subsection~2.3, the {\em nonzero\/} elements of differentiation matrices are
\begin{Eqnarray*}
  \DDD_{(n,m),(k,m)}^{[r]}&=&\begin{case}
    \GG{a}_n\GG{b}_k, & n\geq k+1,\; k+n\mbox{\ odd},\\
    -\GG{a}_k\GG{b}_n, & k\geq n+1,\; k+n\mbox{\ odd},
  \end{case}\\
  \DDD_{(n,m),(k,m)}^{[\theta]}&=&\ii m,\qquad |m|\leq M,\quad k,n=0,\ldots,N.
\end{Eqnarray*}

Finally, we address the issue of linear algebraic calculations. The matrix $\DDD^{[r]}$ being semi-separable and $\DDD^{[\theta]}$ diagonal, it takes just $\O{M+N}$ flops to compute products of differentiation matrices by vectors (i.e.\ the computation of expansion coefficients of derivatives from expansion coefficients of a function) \cite{iserles24ssm}. Semi-separability also lends itself to the solution of linear systems in $\O{M^2}$ operations \cite{hackbusch15hma} and thus also to the computation of a matrix exponential (and other analytic matrix functions) through the Dunford formula
\begin{displaymath}
  g(\mathcal{A})\MM{v}=\frac{1}{2\pi\ii} \oint_{\partial\Omega} g(\lambda) (\lambda I-\mathcal{A})^{-1}\MM{v} \D\lambda,
\end{displaymath}
where $\mathcal A$ is a linear operator and $\Omega$ is a simply-connected domain with Jordan boundary such that $\sigma(\mathcal{A})\subset\Omega$ \cite{hackbusch15hma}. In practical quadrature we discretise the integral at roots of unity, reducing the task in hand to the solution of linear systems with FFT.

\section{The unit ball $\Ball\subset\mathbb{R}^d$, $d\geq2$}

The construction of {\em generalised Zernike functions\/} can be easily extended to any dimension $d\geq2$, since the radial variable $r\in[0,1]$ (we use again  polar coordinates) is always scalar, while there are $d-1$ angular variables $\theta_1\in[-\pi,\pi]$, $\theta_2,\ldots,\theta_{d-1}\in[0,\pi]$. Moreover, we have seen in the last section that the main difficulty is represented by the radial variable and this remains the case for $d\geq3$.

Thus, given $\alpha>0$, we let
\begin{equation}
  \varphi_{n,\Mm{m}}(r,\MM{\theta})=\frac{[r(1-r)]^{\alpha/2}}{\sqrt{2h_n^{(\alpha,\alpha)}}\pi^{(d-1)/2}} \tilde{\PP}_n^{(\alpha,\alpha)}(2r-1)\ee^{\ii(2\Mm{m}^\top\!\Mm{\theta}-m_1\theta_1)},\;\; n\in\mathbb{Z}_+,\; \MM{m}\in\mathbb{Z}^{d-1},
\end{equation}
where $r\in[0,1]$, $\theta_1\in[-\pi,\pi]$, $\theta_\ell\in[0,\pi]$, $\ell=2,\ldots,d-1$, $\tilde{\PP}_n^{(\alpha,\alpha)}$ is the $n$th orthonormal ultraspherical polynomial and $\alpha>0$: in practice we  use $\alpha=2$.

We again factorise $f\in\CC{L}_2(\Ball)$ into a POS, $f=f_0+f_1$, such that both $f_0$ and $f_1$ obey zero boundary conditions for $r=1$, $f_1(0,\MM{\theta})=0$ and the functions are orthogonal with respect to the \R{eq:1.5} scalar product. Computational details change little from the case $d=2$: there is just one radial variable, while calculating the expansion
\begin{displaymath}
  \hat{f}_{n,\Mm{m}}=\int_0^1 \int_{-\pi}^\pi\int_0^\pi\cdots\int_0^\pi  f(r,\MM{\theta})\overline{\varphi_{n,\Mm{m}}(r,\MM{\theta})}\D\theta_1\D\theta_2\cdots\D\theta_{d-1}\D r
\end{displaymath}
for $ n=0,\ldots,N$, $|\MM{m}|\leq M$, can be accomplished with a $(d-1)$-dimensional FFT at the cost of $\O{N^{d-1}\log_2N}$ operations. While assembling  differentiation matrices requires great care, they are again massively sparse, $\DDD^{[r]}$ is semi-separable and the setting lends itself to fast algebra.

\section{Numerical examples}

\subsection{$d=2$}

In this section we approximate the test function
\begin{eqnarray*}
f(r, \theta) = (1-r) \ee^{r} \ee^{\ii (\theta+\frac12)}, \quad \theta \in [-\pi, \pi], \quad r \in [0, 1], \quad f(1, \theta) = 0,
\end{eqnarray*}
which is displayed in Fig.~\ref{Fig:4.1} in polar coordinates, by the series
\begin{eqnarray*}
f_{N, K} = \sum_{n=0}^{N} \sum_{k=-K}^K \hat{f}_{n, k} \varphi_{n, k}(r, \theta),
\end{eqnarray*}
where $\{\varphi_{n, k}(r, \theta)\}$ has been given in \R{eq:2.3}, using the values $N=6$ and $K=5$. We rearrange $\hat{f}_{n, k}$, $n=0, \cdots, N$ and $k = -K, \cdots, K$ into the sequence $\hat{f}_q$ with a single index, $q=0, \cdots, (2K+1)(N+1)$. Considering that many $\hat{f}_q$s are zeros and the sequence is highly sparse, we extract the nonzero coefficients to exhibit the decay rate of the coefficients $\hat{f}_q$.

\begin{figure}[tb]
  \begin{center}
    \includegraphics[width=160pt]{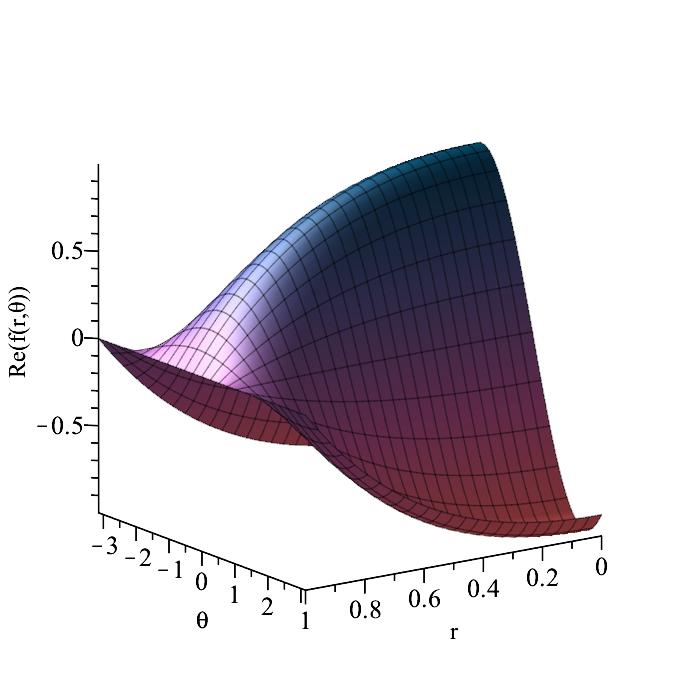}\hspace*{15pt}
    \includegraphics[width=160pt]{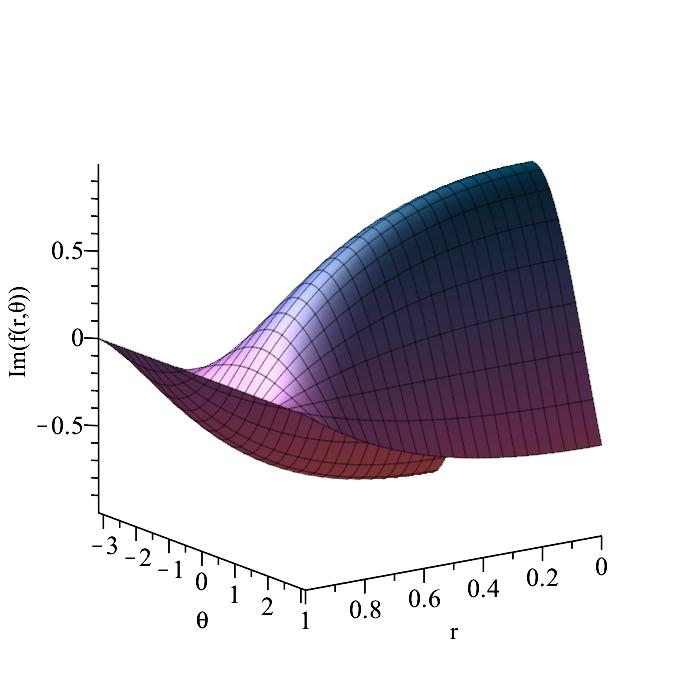}
    \caption{The real (the left) and imaginary (the right) parts of the function $f(r, \theta)$ in the `polar rectangle' $r\in[0,1]$, $\theta\in[-\pi,\pi]$.}
    \label{Fig:4.1}
  \end{center}
\end{figure}

In order to illustrate the behaviour of the error, we define the componentwise forms of $\ell_2$ and $\ell_{\infty}$ errors as
\begin{Eqnarray*}
\MM{e}_{\infty} &:=& \max \left|e(r_m, \theta_j)\right|, \quad \MM{e}_2 := \sqrt{\sum_{j=0}^{M }\sum_{m=0}^{M } \left|e(r_m, \theta_j)\right|^2}, \qquad \mbox{where}\\
r_m &=& \sin^2 \frac{m \pi}{2M }, \quad \theta_j = -\pi + \frac{2 j m\pi}{M}, \qquad m, j = 0, \cdots, M,
\end{Eqnarray*}
and $e=f_{N,K}-f$. Here we take $M  = 6$.

\vspace{6pt}
\noindent{\bf Example 1.} This is the only example where we use the norm \R{eq:1.4},
\begin{displaymath}
  \langle\!\langle f,g\rangle\!\rangle=\int_0^1\int_{-\pi}^\pi r f(r,\theta)\overline{g(r,\theta)}\D \theta\D r
\end{displaymath}
-- in all remaining examples we employ the Cartesian norm \R{eq:1.5} over $[r,\theta]$, $r\in[0,1]$, $\theta\in[-\pi,\pi]$. To this end we set
\begin{displaymath}
  \varphi_{n,m}(r,\theta)= \frac{(2\pi)^{-1/2}}{\sqrt{h_n^{(\alpha,1)}}} (1-r)^{\alpha/2}\PP_n^{(\alpha,1)}(2r-1)\ee^{\ii m\theta},\qquad n\in\mathbb{Z}_+,\quad m\in\mathbb{Z},
\end{displaymath}
where $\alpha>1$: it is trivial that the system $\{\varphi_{n,m}\}_{n\in\mathbb{Z}_+,m\in\mathbb{Z}}$ is an orthonormal basis of $\mathfrak{A}(\Ball)$ with respect to the inner product $\langle\!\langle\,\cdot\,,\,\cdot\,\rangle\!\rangle$. The expansion of a function $f$ is
\begin{displaymath}
  f(r,\theta)=\sum_{n=0}^\infty\sum_{m=-\infty}^\infty \check{f}_{m,n} \varphi_{n,m}(r,\theta),
\end{displaymath}
where
\begin{displaymath}
  \check{f}_{n,m}=\int_0^1\int_{-\pi}^\pi r f(r,\theta) \overline{\varphi_{n,m}(r,\theta)}\D\theta\D r.
\end{displaymath}

It follows from the analysis of Subsection~2.1 (and is verifiable easily by direct computation) that
\begin{displaymath}
  \DDD_{(n,k),(m,k)}^{\,[r]}+\DDD_{(m,k),(n,k)}^{\,[r]}=-\mathcal{S}_{m,n},\qquad m,n\in\mathbb{Z}_+,\quad k\in\mathbb{Z},
\end{displaymath}
where
\begin{displaymath}
  \mathcal{S}_{m,n}=\int_0^1 \varphi_m(r)\varphi_n(r)\D r.
\end{displaymath}
The matrix $\mathcal{S}$ can be computed explicitly using the methodology of \cite{gao23rrg} and the details are left to the reader,
\begin{displaymath}
  \mathcal{S}_{m,n}=(-1)^{m+n} \sqrt{\frac{n+1}{m+1}} \sqrt{\frac{(\alpha+n+1)(\alpha+2m+2)(\alpha+2n+2)}{\alpha+m+1}},\qquad m\geq n,
\end{displaymath}
with symmetric completion for $m\leq n-1$. In particular, $\mathcal{S}$ is nonzero and the departure from skew-symmetry of $\DDD^{\,[r]}$ is substantial.\footnote{Note that $\mathcal{S}$ is semi-separable of rank 2 -- this, of course, has no bearing on our work.}
The magnitude of the coefficients $\check{f}_n$ is displayed in logarithmic scale in Fig.~\ref{Fig:4.2}. The decay is exponential and, purely from the standpoint of approximation theory, this is an excellent means to approximate in the norm $\langle\!\langle\,\cdot\,,\,\cdot\,\rangle\!\rangle$. However, it is of little merit in the context of spectral methods because $\DDD^{\![r]}$ departs so substantially from skew symmetry.

\begin{figure}[tb]
  \begin{center}
    \includegraphics[width=150pt]{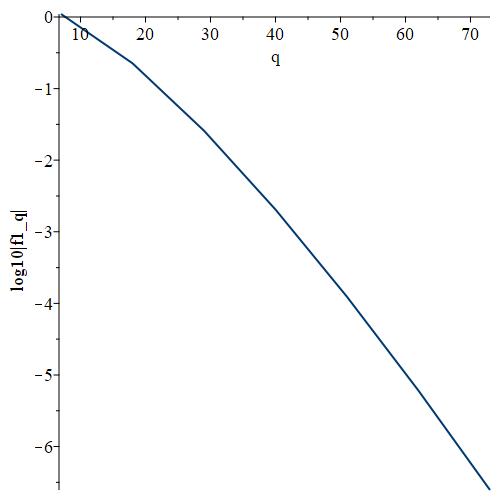}\hspace*{2pt}
    \includegraphics[width=150pt]{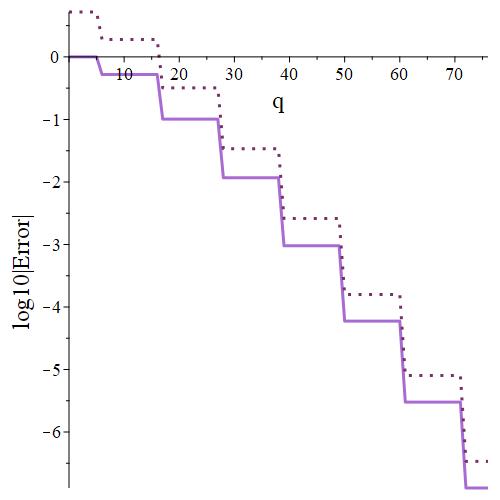}
    \caption{The logarithmic nonzero coefficients $\log10|\check{f}_q|$ and the errors $\MM{e}_\infty$ (solid line) and $\MM{e}_2$ (dotted line) by taking the norm $\langle\!\langle\,\cdot\,,\,\cdot\,\rangle\!\rangle$ with $\alpha=2$, $\beta=1$.}
    \label{Fig:4.2}
  \end{center}
\end{figure}

\vspace{6pt}
\noindent{\bf Example 2.} We apply $\{\varphi_{n, k}(r, \theta)\}$, as given in (\ref{eq:2.3}), with $\alpha=2$ and $\beta=0$, to approximate $f(r, \theta)$ directly in the \R{eq:1.5} scalar product. Note that $\beta=0$ prevents the $\varphi_{n,k}$s from vanishing for $r=0$. The nonzero coefficients $|\hat{f}_q|$ are drawn in the logarithmic scale for $q=7, 18, 29, 40, 51, 62, 73$ and the errors $\MM{e}_2$ and $\MM{e}_\infty$,  are displayed in  Fig.~\ref{Fig:4.3}.
\begin{figure}[tb]
  \begin{center}
    \includegraphics[width=160pt]{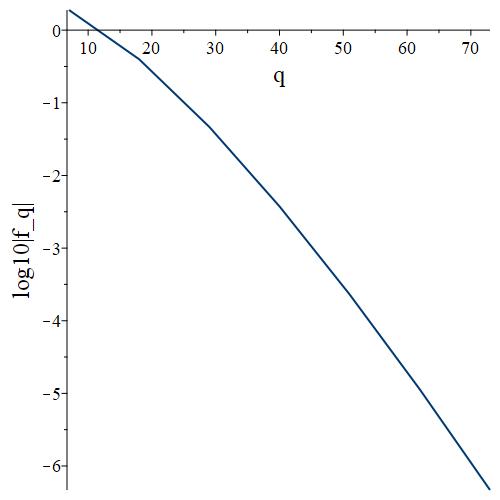}\hspace*{15pt}
    \includegraphics[width=160pt]{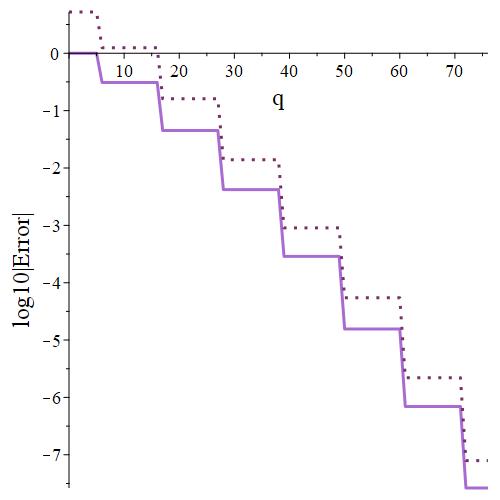}
    \caption{The nonzero coefficients $\hat{f}_q$ (on the left) in logarithmic scale and the errors (on the right) $\MM{e}_\infty$ (solid line) and $\MM{e}_2$ (dotted line) with $\alpha=2$, $\beta=0$.}
    \label{Fig:4.3}
  \end{center}
\end{figure}

It is evident from the figure that convergence takes place at an exponential speed. However, the purpose of the exercise is fast convergence in tandem with (essential) skew symmetry of the differentiation matrices -- and this is impossible in the present case because $\beta=0$. Specifically, for any $\alpha,\beta>-1$, recalling from \R{eq:2.3} that
\begin{displaymath}
  \varphi_{n,k}(r,\theta)=(2\pi)^{-1/2} \tilde{\varphi}_n(r) \ee^{\ii m\theta},\qquad \tilde{\varphi}_n(r)=2^{\frac12(\alpha+\beta+1)} (1-r)^{\alpha/2} r^{\beta/2} \tilde{\PP}_n^{(\alpha,\beta)}(2r-1),
\end{displaymath}
we have for $\alpha\geq0$, $\beta=0$
\begin{Eqnarray*}
  \DDD_{(n,m),(k,k)}+\DDD_{(m,n),(k,k)}&=&\int_0^1 [\tilde{\varphi}_n'(r)\tilde{\varphi}_m(r)+\tilde{\varphi}_n(r)\tilde{\varphi}_m'(r)]\D r\\
  &=&\int_0^1 \frac{\D}{\D r} [\tilde{\varphi}_n(r)\tilde{\varphi}_m(r)]\D r=\tilde{\varphi}_n(1)\tilde{\varphi}_m(1)-\tilde{\varphi}_n(0)\tilde{\varphi}_m(0)\\
  &=&(-1)^{n+m} \sqrt{(\alpha+2n+1)(\alpha+2m+1)}\neq0,
\end{Eqnarray*}
where we have implemented the value of Jacobi polynomials at endpoints from \cite{rainville60sf}. Therefore no skew symmetry is possible -- indeed, $\DDD^{\,[r]}$ is very different from a skew symmetric matrix -- and this course of action confers no advantages compared to Example~1.

\vspace{6pt}
\noindent{\bf Example 3.} The choice $\alpha=\beta>0$ renders $\DDD^{\,[r]}$ skew-symmetric and semi-separable (cf.\ Subsection~2.2), in particular $\alpha=\beta=2$ also means that the $\varphi_{n,k}$s are analytic functions. However, the effect of this choice on the speed of convergence is alarming: because $\varphi_{n,k}(0,\theta)=0$ for all $n\in\mathbb{Z}_+$, $k\in\mathbb{Z}$, we cannot expect pointwise convergence for any $f$ such that $f(0,\theta)\neq0$ and, while $\CC{L}_2$ convergence is possible, it is likely to be excruciatingly slow:  Fig.~\ref{Fig:4.4} is suggestive of $|\hat{f}_q|=\O{q^{-1/4}}$. 

\begin{figure}[tb]
  \begin{center}
    \includegraphics[width=125pt]{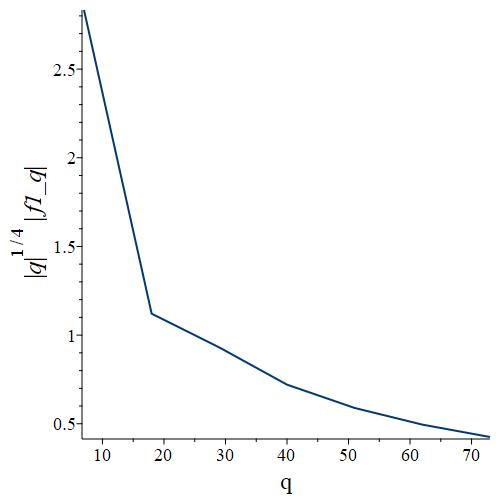}\hspace*{2pt}
    \includegraphics[width=125pt]{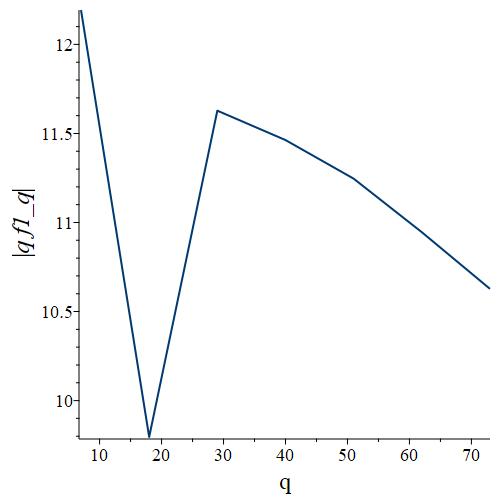}\hspace*{2pt}
    \includegraphics[width=125pt]{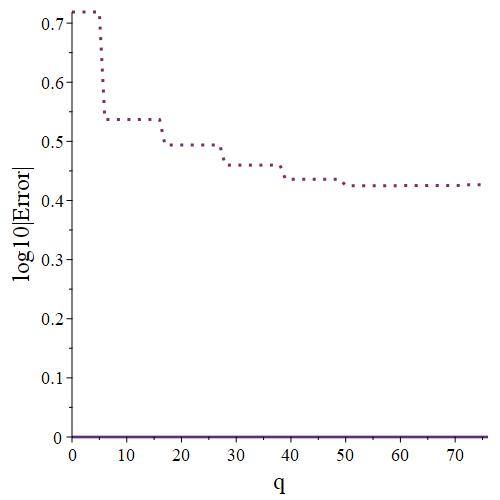}
    \caption{Scaled nonzero coefficients $q^{1/4} |\hat{f}_q|$ (on the left), $q |\hat{f}_q|$ (in the middle) and the errors (on the right) $\MM{e}_\infty$ (solid line) and $\MM{e}_2$ (dotted line) using the basis $\varphi_{n, k}$ directly with $\alpha=\beta=2$.}
    \label{Fig:4.4}
  \end{center}
\end{figure}

\vspace{6pt}
\noindent {\bf Example 4.} As we have seen in Section~2, the right way of resolving the conflict between ensuring fast rate of convergence and generating an essentially skew-symmetric differentiation matrix is to split $f=f_0+f_1$, where $f_0$ is a fixed function, scaled by $f(0)$, while $f_1\in\Azero(\Ball)$ and $f_0\perp f_1$. Our first implementation of splitting takes $\alpha=\beta=1$: this ensures that the differentiation matrix is essentially skew symmetric and semi-separable.

This choice of $\alpha$ and $\beta$, however, leads to an orthogonal system which loses analyticity, because of the factor $\sqrt{r(1-r)}$ in \R{eq:2.3}. This impacts the rate of convergence and this is confirmed in Fig.~\ref{Fig:4.5}. Essentially, the (dismal) rate of convergence is similar to Example~2.

\begin{figure}[tb]
  \begin{center}
    \includegraphics[width=120pt]{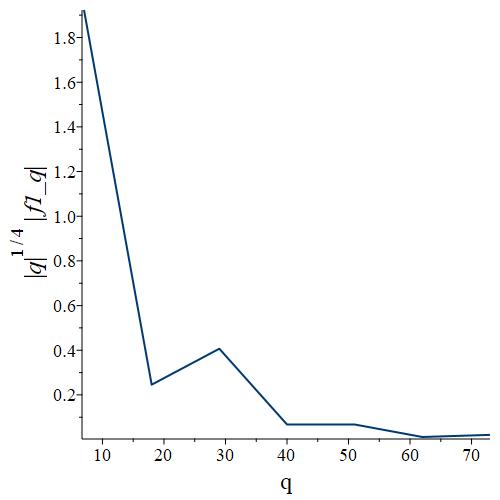}\hspace*{2pt}
    \includegraphics[width=120pt]{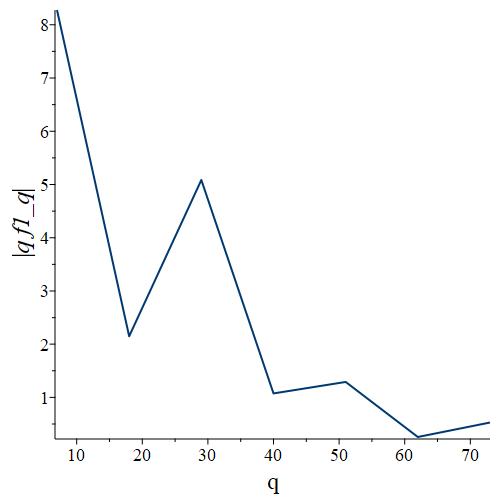}\hspace*{2pt}
    \includegraphics[width=120pt]{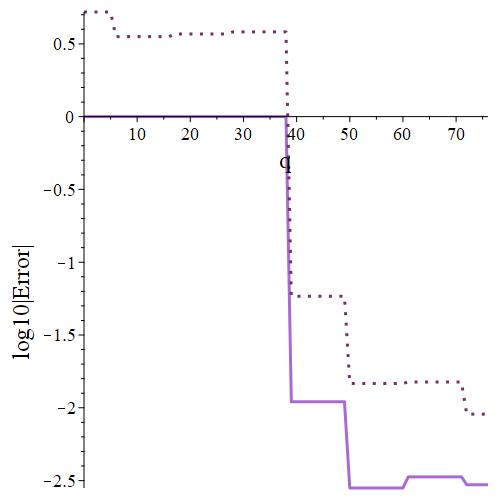}
    \caption{The nonzero coefficients $q^{1/4} |\hat{f}_q|$ (on the left), $q |\hat{f}_q|$ (in the middle) and the errors (on the right) $\MM{e}_\infty$ (solid line) and $\MM{e}_2$ (dotted line) using two orthogonal expansions on $f_0$ and $f_1$ separately for $\alpha=\beta=1$.}
    \label{Fig:4.5}
  \end{center}
\end{figure}

\vspace{6pt}
\noindent{\bf Example 5.}  We now put together all the ingredients that guarantee  an essentially skew-symmetric, semi-separable differentiation matrix and rapid rate of convergence:
\begin{enumerate}
\item POS $f=f_0+f_1$, consistent with the conditions stipulated in Subsection~2.1: $f_0(1,\theta)\equiv f_1(1,\theta)\equiv0$, $f_0(0,\theta)= f(0,\theta)$, $f_1(0,\theta)\equiv0$ and $\langle f_0,f_1\rangle=0$ (in the scalar product \R{eq:1.5}). To be specific,
\begin{displaymath}
  f_0(r,\theta)=(1-r)(2-\ee^r)\ee^{\ii(\theta+\frac12)},\qquad f_1(r,\theta)=2(1-r)(\ee^r-1)\ee^{\ii(\theta+\frac12)};
\end{displaymath}
\item The choice $\alpha=\beta\geq1$, to ensure that differentiation matrices are essentially skew-symmetric and semi-separable;
\item Taking $\alpha=\beta$ as an even natural number, so that the $\varphi_{n,k}$s are analytic functions of $r$ (and of course of $\theta$), thereby ensuring rapid convergence for analytic (or very regular) functions $f$.
\end{enumerate}
Differentiation matrices are essentially skew symmetric, in line with the analysis in Subsection~2.3. They are also semi-separable of rank 2, allowing for efficient linear algebra.

\begin{figure}[tb]
  \begin{center}
    \includegraphics[width=160pt]{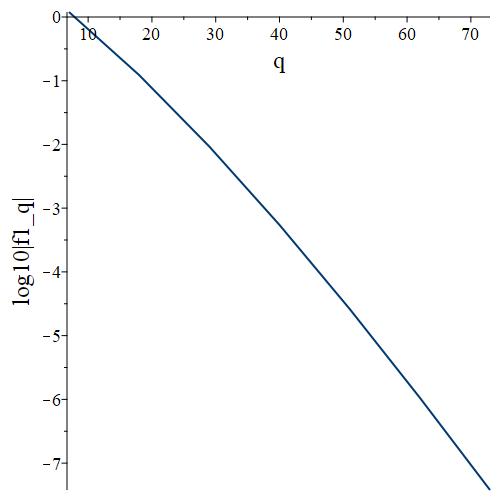}\hspace*{15pt}
    \includegraphics[width=160pt]{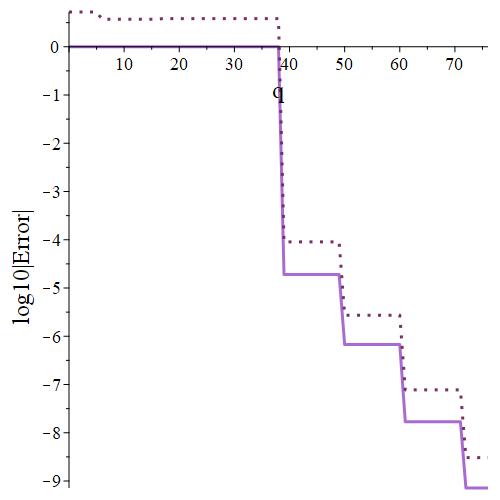}
    \caption{The nonzero coefficients in logarithmic scale (on the left) and the errors (on the right) $\MM{e}_\infty$ (solid line) and $\MM{e}_2$ (dotted line) using two orthogonal expansions on $f_0$ and $f_1$ separately for $\alpha=\beta=2$.}
    \label{Fig:4.6}
  \end{center}
\end{figure}

Fig.~\ref{Fig:4.6} displays the rate of decay of $|\hat{f}_q|$ and the errors $\MM{e}_\infty$ and $\MM{e}_2$ in logarithmic scale to base 10. It is clear that geometric decay takes place in all cases -- the decay of $\MM{e}_2$ and $\MM{e}_\infty$ appears to be a step function because the vector $\hat{\MM{f}}$ is highly sparse. Note that we obtain accuracy of $\approx 10^{-9}$ with just 77 degrees of freedom {\em in two dimensions\/}, which is quite remarkable! Note further that the rate of decay of the coefficients is substantially faster than in Fig.~\ref{Fig:4.2}: thus, we get both better differentiation matrix and smaller error! Of course, the $\check{f}_n$s in Example~1 and the $\hat{f}_n$s here correspond to different norms but it is trivial that $\langle\!\langle f,f\rangle\!\rangle\leq \langle f,f\rangle$ so, intuitively, were the error in the two expansions to decay at similar speed, we would expect the $\check{f}_n$s to decay {\em faster!\/} Inasmuch as not much can be read into a single example, we regard this  as worthy of a comment.

\begin{figure}[tb]
  \begin{center}
    \includegraphics[width=160pt]{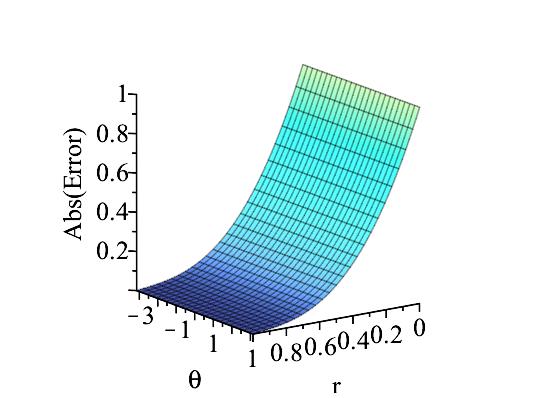}\hspace*{15pt}
    \includegraphics[width=160pt]{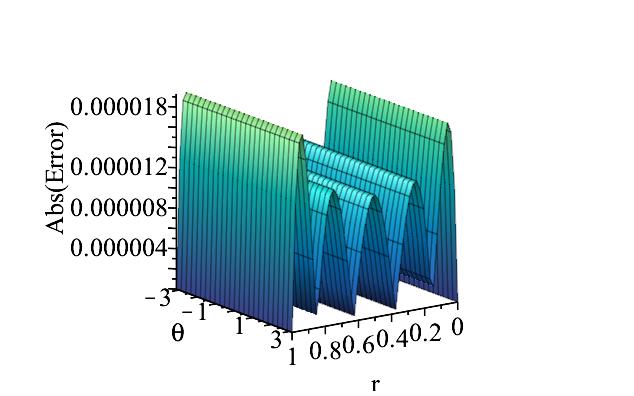}\\
    \includegraphics[width=160pt]{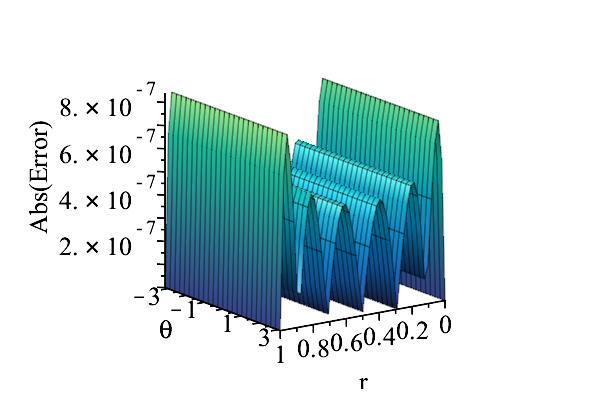}\hspace*{15pt}
    \includegraphics[width=160pt]{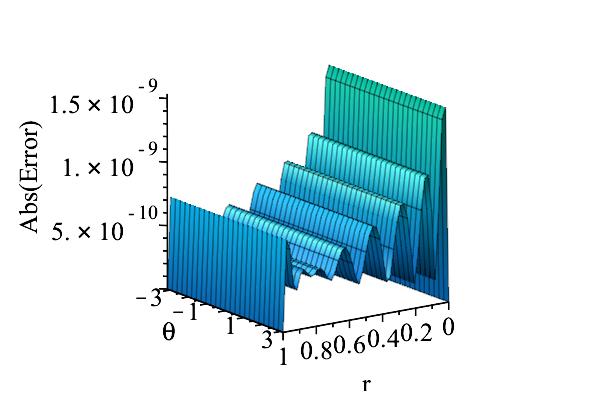}\\
    \caption{Pointwise errors in the $(r,\theta)$ rectangle using the basis from Example~5 with $q=10, 40, 50, 76$ (from the top left, top right, bottom left and bottom right) and $\alpha = \beta = 2$.}
    \label{Fig:4.7}
  \end{center}
\end{figure}

Finally, in Fig.~\ref{Fig:4.7} we display pointwise error in approximating $f$ with the current W-functions, in tandem with POS: the accuracy is remarkable. Just 77 degrees of freedom yield an error of $\approx 10^{-9}$ in two dimensions!

\subsection{$d=3$}

We provide a single example in three dimensions, using the same methodology as in Example~5: to recall, this represented the combination of all necessary ingredients ensuring both exponential convergence and an essentially skew-symmetric differentiation matrices. We consider the function
\begin{equation}
  \label{eq:4.1}
  f(r,\theta_1,\theta_2)=(1-r)\ee^r \ee^{\ii(\frac12+\theta_1+2\theta_2)}=f_0(r,\theta_1,\theta_2)+f_1(r,\theta_1,\theta_2)
\end{equation}
for $r\in[0,1]$, $\theta_1\in[-\pi,\pi]$ and $\theta_2\in[0,\pi]$, where
\begin{Eqnarray*}
  f_0(r,\theta_1,\theta_2)&=&2(1-r)(1-\ee^r)\ee^{\ii(\frac12+\theta_1+2\theta_2)},\\
  f_1(r,\theta_1,\theta_2)&=&2(1-r)(\ee^r-2)\ee^{\ii(\frac12+\theta_1+2\theta_2)}
\end{Eqnarray*}
is a POS. Specifically,
\begin{displaymath}
  f(r,\theta_1,\theta_2)=\sum_{k_1,k_2=-K}^K  f^\circ_{n,k_1,k_2}\ee^{\ii(k_1\theta_1+2k_2\theta_2)}+\sum_{n=0}^N\sum_{k_1,k_2=-K}^K \hat{f}_{n,k_1,k_2} \varphi_{n,k_1,k_2}(r,\theta_1,\theta_2),
\end{displaymath}
where
\begin{displaymath}
  \varphi_{n,k_1,k_2}(r,\theta_1,\theta_2)=\frac{r(1-r)}{\sqrt{2h_n^{(2,2)}}\pi} \tilde{\PP}_n^{(2,2)}(2r-1)\ee^{\ii(k_1\theta_1+2k_2\theta_2)}
\end{displaymath}
and
\begin{Eqnarray*}
  f^\circ_{n,k_1,k_2}&=&\frac{1}{2\pi^2} \int_{-\pi}^\pi\int_0^\pi f_0(r,\theta_1,\theta_2) \ee^{-\ii(k_1\theta_1+2k_2\theta_2)}\D\theta_2\D\theta_1,\\
  \hat{f}_{n,k_1,k_2}&=&\int_0^1\int_{-\pi}^\pi \int_0^\pi f_1(r,\theta_1,\theta_2) \overline{\varphi_{n,k_1,k_2}(r,\theta_1,\theta_2)}\D\theta_2\D\theta_1\D r
\end{Eqnarray*}
and $K=3$, $M=6$, $N=5$. We again arrange the $\hat{f}_{n,k_1,k_2}$ into an exceedingly sparse vector $(f_q)$ for $q=0,\ldots,(2K+1)^2(N+1)=294$ and use the same notation $\MM{e}_2$ and $\MM{e}_\infty$ for errors as in the two-dimensional case, computing them at the discrete points
\begin{displaymath}
  r_m=\sin^2\frac{m\pi}{2M},\quad\;\; \theta_{1,j}=-\pi+\frac{2jm\pi}{M},\quad \theta_{2,j}=\frac{jm\pi}{M},\qquad m,j=0,\ldots,M.
\end{displaymath}

\begin{figure}[tb]
  \begin{center}
    \includegraphics[width=160pt]{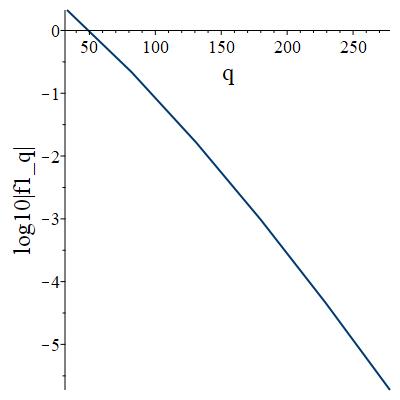}\hspace*{15pt}
    \includegraphics[width=160pt]{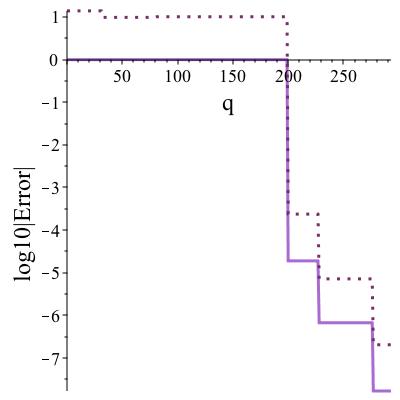}
    \caption{The nonzero coefficients in logarithmic scale (on the left) and the errors (on the right) $\MM{e}_\infty$ (solid line) and $\MM{e}_2$ (dotted line) in three dimensions.}
    \label{Fig:4.8}
  \end{center}
\end{figure}

In Fig.~\ref{Fig:4.8} we display the decay of the nonzero coefficients and the errors for the function \R{eq:4.1}. Exponential decay is evident. Moreover, as follows from our theory,   differentiation matrices are essentially skew symmetric.

\section*{Acknowledgements} JG's work has been  supported by by Natural Science Basic Research Plan in Shaanxi Province of China (Grant No.\ 2023JM-026) and National Natural Science Foundation of China (Grant No.\ 12271426).

\bibliographystyle{agsm}

\end{document}